\documentclass{amsart}
\pdfoutput=1
\usepackage{graphics}
\usepackage{nicefrac}
\usepackage[english]{babel}
\usepackage[utf8]{inputenc}
\usepackage{graphicx}
\usepackage[hang,small,sc]{caption}
\usepackage{calc}
\usepackage{float}
\restylefloat{figure}

\pdfoutput=1
\usepackage{thumbpdf}
\usepackage[dvipsnames]{xcolor}
\usepackage[colorlinks=true,linkcolor=MidnightBlue,citecolor=Sepia]{hyperref}
\usepackage{amsmath}
\usepackage{amsfonts}
\usepackage{amssymb}
\usepackage{amstext}
\usepackage{subfigure}

\newcommand{\R}{\mathbb{R}}

\newcommand{\N}{\mathbb{N}}
\newcommand{\B}{Y}

\newcommand{\cB}{{\mathcal{B}}}
\newcommand{\cC}{{\mathcal{C}}}

\def\diam{\mathop{\rm diam}\nolimits}

\newtheorem{theorem}{Theorem}[section]

\newtheorem{lemma}[theorem]{Lemma}
\newtheorem{proposition}{Proposition}

\theoremstyle{definition}
\newtheorem{definition}[theorem]{Definition}
\newtheorem{remark}{Remark}

\title{On the Computation of Attractors for Delay Differential Equations}
\author[M. Dellnitz, M. Hessel-von Molo A. Ziessler]{Michael Dellnitz, Mirko Hessel-von Molo and Adrian Ziessler}
\email{dellnitz@uni-paderborn.de}
\email{mirkoh@uni-paderborn.de}
\email{ziessler@uni-paderborn.de}

\date{}

\begin{document}

\begin{abstract}
\noindent In this work we present a novel framework for the computation of
finite dimensional invariant sets of infinite dimensional dynamical
systems. It extends a classical subdivision technique \cite{DH97} for the computation of
such objects of finite dimensional systems to the infinite
dimensional case by utilizing results on embedding techniques for
infinite dimensional systems. We show how to implement this approach for
the analysis of delay differential equations and illustrate the
feasibility of our implementation by computing invariant sets for
three different delay differential equations.  
\end{abstract}

\maketitle

\centerline{\scshape M.~Dellnitz, M.~Hessel-von Molo and A.~Ziessler}
\medskip
{\footnotesize
\centerline{Chair of Applied Mathematics}
\centerline{University of Paderborn}
\centerline{33095 Paderborn, Germany}
}

\section{Introduction}

Over the last two decades so-called \emph{set oriented numerical methods}
have been developed in the context of the numerical treatment of dynamical systems
(e.\ g.\ \cite{DH97, DJ99, FD03, FLS10}).
The basic idea is to cover the objects of interest -- for instance \emph{invariant sets}
or \emph{invariant measures} -- by outer approximations which are created via multilevel
subdivision techniques. These techniques have been used in several different
application areas such as molecular dynamics (\cite{SHD01}), astrodynamics (\cite{DJLMPPRT05}) or
ocean dynamics (\cite{FHRSSG12}).

So far the applicability of the subdivision scheme is restricted to
finite dimensional dynamical systems, i.\ e.\ ordinary differential equations or
finite dimensional discrete dynamical systems. In this paper we extend
this technique to the infinite dimensional context. More precisely, we
develop a set oriented numerical technique which allows us to compute
low-dimensional invariant sets for infinite dimensional dynamical systems.
Rather than using a straightforward approach based on an appropriate
combination of Galerkin expansions and subdivision steps we follow
a completely novel path and utilize \emph{embedding results} in our
numerical treatment.

The first result on embeddings in the dynamical systems context is
the celebrated \emph{Takens Embedding Theorem} \cite{T81}. Takens
has shown that an invariant set -- in his context this
has to be a compact manifold of dimension $d$ -- can generically be
reconstructed using the so-called \emph{observation map} which consists of
observations of the dynamical behavior at an appropriate number (at least $2d+1$) of consecutive
snapshots in time.
This result has been extended by Sauer et al.\ in \cite{embedology} to the context of
compact invariant sets of box counting dimension $d$. There it has been
shown that the same observation map can be used for the reconstruction
of the invariant set
as long as more than $2d$ consecutive snapshots in time are used.
Moreover in this work the notion of ``genericity" has been replaced by the
more intuitive notion of ``prevalence". Finally, in \cite{R05} Robinson extended the
results obtained in \cite{embedology} to dynamical systems on infinite
dimensional Banach spaces. It turns out that also here the same observation map can be used
to reconstruct invariant sets of finite dimensional box counting dimension $d$.
However, in addition to the dimension of the set another quantity comes into play,
namely the \emph{thickness exponent} $\sigma$. Roughly speaking this exponent measures
how well the invariant set can be approximated by finite dimensional
subspaces of the underlying Banach space. The lower bound $2d$ on the number
of snapshots has to be replaced by $2(1+\sigma)d$ accordingly.

We remark
that there are several further extensions of Takens' theorem. For instance in
\cite{Stark:99} forced systems are considered, and in
\cite{mezic:comparisoncomplexbehav} one can find a stochastic version of
this result.

Our results in this paper are based on Robinson's embedding theorem. In fact, we will combine
the reconstruction based on the observation map with the classical subdivision
techniques developed in \cite{DH97}. Assuming that a bound on the box-counting
dimension and the thickness exponent of the invariant set are known we use the
observation map and its inverse to define a dynamical system $\varphi$ in the embedding
space of dimension $k > 2(1+\sigma)d$. Then the subdivision scheme is 
applied to compute the reconstructed invariant set for $\varphi$. Observe that this way we can
always perform the numerical approximation within a finite dimensional space of fixed
dimension $k$, and this is in contrast to Galerkin based approaches where one would
have to extend the expansions in order to improve the quality of the approximation.

The numerical approach we are proposing is in principle applicable to arbitrary
infinite dimensional dynamical systems. However, here we will restrict our
attention to delay differential equations with constant delay in the numerical
realization. The applications to e.\ g.\ partial differential equations will be done in
future work.

A detailed outline of the paper is as follows. In Section~\ref{sec:takens} we briefly summarize
the infinite dimensional embedding theory introduced in \cite{R05}. 
In Section~\ref{sec:comp_emb_att} we employ the main result of \cite{R05}
for the construction of a numerical approach for the computation of compact
finite dimensional attractors of infinite dimensional dynamical systems. First we construct
a continuous dynamical system $\varphi$ on the embedding space using a generalization
of the well-known Tietze extension theorem \cite[I.5.3]{ds} which is due to
Dugundji \cite[Theorem~4.1]{Dugundji51}. Then we
extend the results from \cite{DH97} to the situation
where the underlying dynamical system $\varphi$ is just continuous (and not
homoemorphic). A numerical realization for the computation of attractors of
delay differential equations is given in Section~\ref{sec:num_real_DDE}.
Finally, in Section~\ref{sec:numex}, we illustrate the efficiency of our novel
approach for three different delay differential equations.

\section{Infinite Dimensional Takens Embedding}\label{sec:takens}

We start with a short review of the contents of \cite{R05}.
We consider a dynamical system of the form

\begin{equation}\label{eq:DS}
u_{j+1} = \Phi (u_j),\quad j=0,1,\ldots,
\end{equation} 

where $\Phi:\B\rightarrow \B$ is Lipschitz continuous on a Banach space $\B$.
Moreover we assume that $\Phi$ has an invariant compact set $A$, that is
\[
\Phi(A) = A.
\]

Later on we will additionally assume that $A$ is a global attractor in the sense that
it attracts all bounded sets within $\B$ as $t\rightarrow \infty$.

For the statement of the main result of \cite{R05} we need three
particular notions:
prevalence \cite{embedology}, upper box counting dimension and thickness
exponent \cite{HuntKaloshin99}.

\begin{definition}[Prevalence]
	A Borel subset $S$ of a normed linear space $V$ is
	\emph{prevalent} if there is a finite dimensional subspace $E$ of
	$V$ (the `probe space') such that for each $v \in V,\ v+e$ belongs
	to $S$ for (Lebesgue) almost every $e\in E$.
\end{definition}

\begin{definition}[Upper box counting dimension] 
	Let $\B$ be a Banach space, and let $X\subset \B$ be compact.
	For $\varepsilon >0$, 
	denote by $N_\B(X,\varepsilon)$ the minimal number of balls of radius
	$\varepsilon$ (in the norm of $\B$) necessary to cover the set
	$X$. Then 
	\begin{equation} 
		d(X; \B) = \limsup\limits_{\varepsilon \rightarrow 0} \frac{\log
    N_\B(X,\varepsilon)}{-\log \varepsilon}
    = \limsup_{\varepsilon\to 0}\; -\log_{{\varepsilon}} N_\B(X,\varepsilon)
 \end{equation} 
 denotes the upper box-counting dimension of $X$.
\end{definition}

\begin{definition}[Thickness exponent] 
 	Let $\B$ be a Banach space, and let $X\subset \B$ be compact.
	For $\varepsilon >0$, 
	denote by $d(X, \varepsilon)$ the minimal
	dimension of all finite dimensional subspaces $V\subset \B$ such
	that every point of $X$ lies within distance $\varepsilon$ of $V$;
	if no such $V$ exists, $d(X, \varepsilon) = \infty$. Then
	\[
		\sigma(X, \B) := \limsup_{\varepsilon \to 0}
		\;-\log_{{\varepsilon}} d(X, \varepsilon)
	\] 
	is called the \emph{thickness exponent} of $X$ in $\B$. 
\end{definition}

From this definition one sees that essentially, $\sigma(X, \B)$ captures how
well $X$ can be approximated from within finite dimensional subspaces of
$\B$. In \cite{KukavicaRobinson2004}, 
as a more practical expression Kukavica and Robinson prove  that
\[ 
	\sigma(X, \B) = \limsup_{n\to\infty} \frac{\log n}{-\log
	\varepsilon_\B(X, n)},
	\quad\text{i.e. }
	\varepsilon_\B(X,n) \sim n^{-1/\sigma(X, \B)}
\]
where $\varepsilon_\B(X,n)$ is the minimum distance between $X$ and any
$n$-dimensional linear subspace of $\B$. 

These notions are essential in answering the question when a delay
embedding technique applied to an invariant subset $A\subset \B$ will work
generically. More precisely, the results are as follows.

\begin{theorem}[\cite{HuntKaloshin99}]
	Let $\B$ be a Banach space and $X\subset \B$ compact, with upper
	box counting dimension $d(X; \B) =: d$ and thickness exponent
	$\sigma(X, \B) =: \sigma$. Let $N>2d$ be an integer, and let $\alpha
	\in\R$ with
	\[ 0 < \alpha < \frac{ N-2d }{ N\cdot(1+\sigma) }. \]
	Then for a prevalent set of linear maps $\mathcal{L}:\B\to\R^N$ there
	is $C>0$ such that 
	\[ C\cdot \| \mathcal{L}(x-y)\|^\alpha \geq \|x-y\| \quad\text{for all }x,
	y\in X. \]
	\label{thm:HK99}
\end{theorem}

Note that this implies that a prevalent set of linear maps
$\B\to\R^N$ will be one-to-one on $X$. Using this theorem, the following
result concerning the delay embedding technique can be proven.

\begin{theorem}[\cite{R05}]
	\label{thm:R05}
Suppose that the upper box counting dimension of $A$ is $d(A)=d$, and that
$A$ has a thickness exponent $\sigma$. Choose an integer $k> 2(1+\sigma)d$
and suppose further that the set $A_p$
of $p$-periodic points of $\Phi$ satisfies $d(A_p) < p/(2+2\sigma)$ for $p=1,\ldots,k$.
Then for a prevalent set of Lipschitz maps $f:\B\rightarrow \R$ the observation map
$D_k[f,\Phi] : \B \rightarrow \R^k$ defined by
\begin{equation}
	\label{def:D_k} 
	D_k[f,\Phi] (u) = (f(u),f(\Phi(u)),\ldots,f(\Phi^{k-1}(u)))^T
\end{equation}
is one-to-one on the invariant set $A$. 
\end{theorem}

\begin{remark}
	\label{rmk:ext_R05}
	Following an observation already made in
	\cite[Remark 2.9]{embedology}, we note that this result may be generalized to
	the case where several independent observables are evaluated. More
	precisely, also for a prevalent set of Lipschitz maps
	$f_1, \dots, f_q:\B\to\R$ the observation map $D_k[f_1, \ldots, f_q]:
	\B\to\R^k$,
	\begin{equation}
		u \mapsto (%
		f_1(u), \ldots, f_1(\Phi^{k_1-1}(u)),%
		\ldots, %
		f_q(u), \ldots, f_q(\Phi^{k_q-1}(u)))^T
		\label{def:D_kvar}
	\end{equation}
	is one-to-one on $A$, provided that $ k =\sum_{i=1}^q k_i >
	2(1+\sigma)\cdot d$ and $d(A_p) < p/(2+2\sigma)$ for $p\leq \max(k_1,
	\ldots, k_q)$.
\end{remark}

\section{Computation of Embedded Attractors via Subdivision}
\label{sec:comp_emb_att}

\subsection{Finite-dimensional Embeddings of Attractors}
\label{ssec:findim_emb_att}
In this section we employ Theorem \ref{thm:R05} in order to construct a
method for the computation of compact, finite dimensional attractors of
infinite dimensional dynamical systems on a Banach space $\B$. 

Let us denote by $A_k$ the image of $A\subset \B$ under $D_k[f,\Phi]$, that is
\[
A_k = D_k[f,\Phi] (A),
\]
where $D_k$ is the map defined in Theorem \ref{thm:R05} and $f$ is chosen
such that $D_k$ is one-to-one on $A$.

We now develop a set oriented numerical technique for the
approximation of the set $A_k$. First, we define a dynamical system
on $\R^k$ for which $A_k$ is an invariant set, on which the dynamics is
conjugate to that of $\Phi$ on $A$.
For this we define the map $\varphi : \R^k \rightarrow \R^k$ by
\begin{equation}
	\label{def:phi}
\varphi = R \circ \Phi \circ E, 
\end{equation}
where $E:\R^k \rightarrow \B$ and $R: \B \rightarrow \R^k$ are an embedding 
and a restriction, respectively, that satisfy 
\begin{equation} 
	\label{eq:condRE} 
	(E\circ R)(u) = u\quad \forall u\in A
	\quad\text{ and }\quad
	(R\circ E)(x) = x\quad \forall x\in A_k.
\end{equation}
More concretely, we define the map $R=D_k$ by the right-hand side of
(\ref{def:D_k}) (see Figure~\ref{fig:kommdia}).

\begin{figure}[H]
    \centering
        \includegraphics[width=0.5\textwidth]{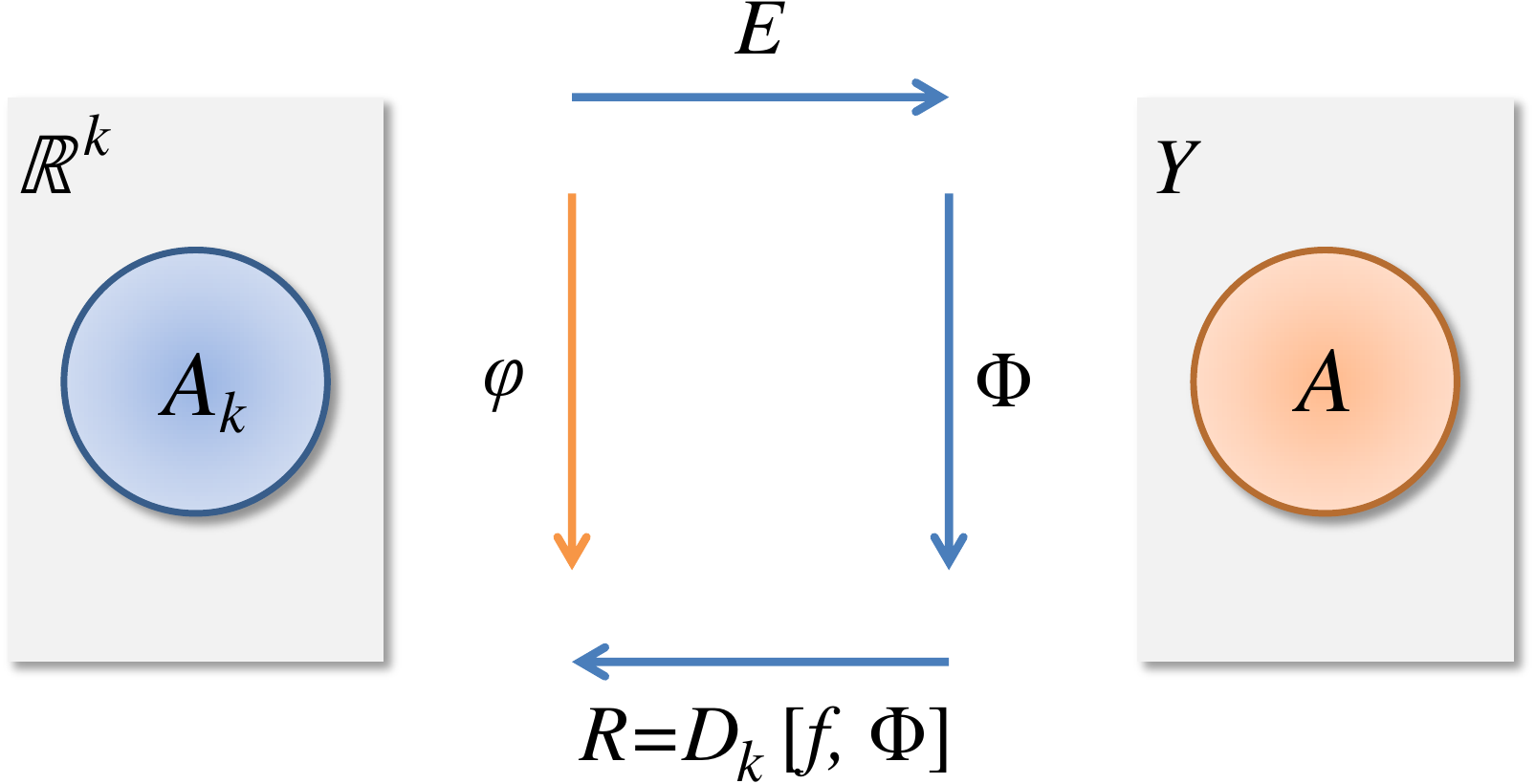} 
    \caption{Definition of the map $\varphi$.}
    \label{fig:kommdia}
\end{figure}

\begin{remark}
Although $D_k$ is derived from what would commonly be termed an embedding
theorem, we call $R$ a restriction, as it maps from an
infinite dimensional domain into a finite-dimensional set, and $E$ an
embedding, as it maps (or embeds) a finite-dimensional into an
infinite dimensional space.
\end{remark}

The map $E$ is obtained in two steps: Firstly, as $R$ is one-to-one on $A$, 
the requirement that
\begin{equation}
	R\circ \tilde E(x)  = x \text{ for all }x\in A_k
	\label{eq:ER}
\end{equation} 
uniquely defines a map $\tilde E:A_k\to A$.
In a second step, we extend this map to a
continuous map $E:\R^k\to \B$. To do this, we employ a generalization 
of the well-known Tietze extension theorem \cite[I.5.3]{ds} found by
Dugundji \cite[Theorem~4.1]{Dugundji51}, stated here with notation adapted
to our needs:

\begin{theorem}
	\label{thm:Dugundji}
	Let $X$ be an arbitrary metric space and $\mathcal{A}\subset X$ be closed,
	$V$ a locally convex linear space and $p:\mathcal{A}\to V $ be continuous. 
	Then there is a continuous map $P:X\to V$ with $P_{|\mathcal A} = p$ such that 
	$P(X)$ is contained in the convex hull of $p(\mathcal{A})$. 
\end{theorem}

Using this theorem in the situation introduced above, we obtain the
following 

\begin{proposition}
	\label{prop:phi_is_cont}
	There is a continuous map $\varphi:\R^k \to \R^k$ satisfying
	\begin{equation}
		\varphi(R(u)) = R(\Phi(u)) \text{ for all } u\in A. 
		\label{eq:conjugonA}
	\end{equation}
\end{proposition}
\begin{proof}
	By construction, the map $R:\B\to \R^k$ given by (\ref{def:D_k}) is
	continuous (even Lipschitz) and one-to-one. Thus, restricting $R$
	to $A$ we obtain a bijective map $\tilde R: A\to A_k$. As $A$ is
	assumed to be compact and $A_k \subset \R^k$ is Hausdorff, $\tilde
	R$ is a homeomorphism by a well-known theorem from elementary
	topology (see e.~g. \cite[Theorem 17.14]{WillardTopology}). Thus
	we obtain a continuous map $\tilde E:A_k\to A$ as $\tilde E =
	\tilde R^{-1}$. 

	As $\B$ is a normed linear space, it is locally convex. Thus we can
	apply Dugundji's Theorem with $X=\R^k$, $\mathcal A=A_k$, $p =
	\tilde E$ and $V= \B$ to see that there is a continuous map
	$E:\R^k\to \B$ with $E_{|A_k} = \tilde E$. Finally, defining
	$\varphi$ through ~(\ref{def:phi}) we obtain that $\varphi$ is
	continuous as a composition of continous maps. 
\end{proof}

Now we are in a position to approximate the embedded invariant set $A_k$
via the corresponding dynamical system
\begin{equation} 
	x_{j+1} = \varphi(x_j)\quad j=0,1,\ldots .
\end{equation} 
To this end, we employ a subdivision scheme as defined in \cite {DH97}.  

\subsection{Subdivision Scheme}
\label{ssec:subdiv}
We briefly review the classical subdivision procedure. 
Let $Q \subset \R^k$ be a compact set. We define the
\emph{global attractor relative to} $Q$ by
\begin{equation}
    \label{eq:relativeAttractor}
    A_Q = \bigcap_{j\ge 0} \varphi^j(Q).
\end{equation}
The subdivision procedure allows us to approximate this set. Roughly
speaking, the idea of the algorithm is as follows.  We start with a finite
family of (large) compact subsets of $\R^k$ which cover the domain in
which we want to analyze the dynamical behavior. Then we subdivide each of
these sets into smaller ones and throw away subsets which do not contain
part of the relative global attractor.  Continuing the process with the
new collection of (smaller) sets it becomes intuitively clear that this
should lead to a successively finer approximation of the relative global
attractor.

Let us be more precise.
The algorithm generates a sequence $\cB_0,\cB_1,\ldots$ of finite
collections of compact subsets of $\R^k$ such that the diameter
\[
\diam(\cB_\ell) = \max_{B\in\cB_\ell}\diam(B)
\]
converges to zero for $\ell\rightarrow\infty$.  Given an initial
collection $\cB_0$, we inductively obtain $\cB_\ell$ from $\cB_{\ell-1}$
for $\ell=1,2,\ldots$ in two steps.

\begin{enumerate}
\item {\em Subdivision:} Construct a new collection 
$\hat\cB_\ell$ such that
  \begin{equation}
    \label{eq:sd1}
    \bigcup_{B\in\hat\cB_\ell}B = \bigcup_{B\in\cB_{\ell-1}}B
  \end{equation}
  and
  \begin{equation}
    \label{eq:sd2}
    \diam(\hat\cB_\ell) = \theta_\ell\diam(\cB_{\ell-1}),
  \end{equation}
  where $0<\theta_{\min} \le \theta_\ell\le \theta_{\max} < 1$.
\item {\em Selection:} Define the new collection $\cB_\ell$ by
  \begin{equation}
    \label{eq:select}
    \cB_\ell=\left\{B\in\hat\cB_\ell : \exists \hat B\in\hat\cB_\ell
      ~\mbox{such that}~\varphi^{-1}(B)\cap\hat B\ne\emptyset\right\}.
  \end{equation}
\end{enumerate}

The first step guarantees that the collections $\cB_\ell$ consist of
successively finer sets for increasing $\ell$.  In fact, by construction
\[
        \diam(\cB_\ell)\leq\theta_{\max}^\ell\diam(\cB_0)\rightarrow 0\quad 
        \mbox{for $\ell\rightarrow\infty$.}
\]
In the second step we remove each subset whose preimage does neither
intersect itself nor any other subset in $\hat\cB_\ell$.  As we shall see,
this step is responsible for the fact that the unions
$\bigcup_{B\in\cB_\ell}B$ approach the relative global attractor.

Denote by $Q_\ell$ the collection of compact subsets obtained after $\ell$
subdivision steps, that is, 
\[ 
	Q_\ell=\bigcup_{B\in\cB_\ell}B.  
\]
Moreover let $\cB_0$ be a finite collection of closed subsets with
$Q_0=\bigcup_{B\in\cB_0}B=Q$. Then the main convergence result
of \cite{DH97} states that
  \[
  \lim\limits_{k\rightarrow\infty}
  h\left( A_Q, Q_k \right)=0,
  \]
where $h(B,C)$ is the usual Hausdorff distance between two
compact subsets $B,C\subset\R^n$. However, in that work the authors
assume that $\varphi$ is a homeomorphism and not just continuous, as in
the situation here. For this reason, in the following we present a proof
of convergence for continuous $\varphi$.

\subsection{Proof of Convergence}

Essentially, we will be able to follow
the structure of the proof in \cite{DH97}. However, there are some technical
differences, and we will need one additional assumption on $A_Q$.
  
We begin with the following observation:

\begin{lemma}\label{lem:subset}
Suppose that $B\subset Q$ satisfies $B\subset \varphi(B)$. Then $B\subset A_Q$.
\end{lemma}
\begin{proof}
From $B\subset \varphi(B)$ it follows that $\varphi^j(B) \subset \varphi^{j+1}(B)$ for all
$j\ge 0$. Hence
\[
B = \bigcap_{j\ge 0} \varphi^j(B) \subset \bigcap_{j\ge 0} \varphi^j(Q) =A_Q.
\]
\end{proof}
We now can prove our first result.

\begin{proposition}
  Let $A_Q$ be the global attractor relative to the compact set $Q$,
  and suppose that the embedded attractor $A_k$ satisfies $A_k \subset Q$.
  Then
  \begin{equation}
    A_k \subset A_Q.
  \end{equation}
\end{proposition}

\begin{proof} By construction of our dynamical system (see
	\eqref{def:phi}--\eqref{eq:ER}),
we have $\varphi(A_k) = A_k$. Thus, Lemma~\ref{lem:subset} implies that
$A_k \subset A_Q$.
\end{proof}

\begin{remark}
Observe that we can in general not expect that  $A_k = A_Q$. In fact,
by construction $A_Q$ may contain several invariant sets and related
heteroclinic connections. In this sense $A_k$ will be embedded in $A_Q$.
\end{remark}

Next observe that the $Q_\ell$'s define a nested sequence of compact sets, that is, $Q_{\ell+1}\subset Q_\ell$.
Therefore, for each $m$,
\begin{equation}
 Q_m = \bigcap\limits_{\ell =1}^{m} Q_\ell,
\end{equation}
and we may view
\begin{equation}
 Q_{\infty} = \bigcap\limits_{\ell =1}^{\infty} Q_\ell
\end{equation}
as the limit of the $Q_\ell$'s.

Now we will prove the convergence of the subdivision scheme
for continuous $\varphi$. More precisely we will show that $Q_\infty = A_Q$.
This will be done in two steps. The first is
\begin{lemma}
	\label{lem:Qinf_in_phiQinf}
	\[
	Q_\infty \subset A_Q.
	\]
\end{lemma}
\begin{proof}
	We will show that
	\[
	Q_\infty \subset \varphi(Q_\infty).
	\]
	Then the result follows with Lemma \ref{lem:subset}.
	
	Let $y\in Q_\infty$. Then for every $\ell \ge 0$ there is a unique
	$B_\ell(y)\in \cB_\ell$ with $y\in B_\ell(y)$. By the selection step of
	the subdivision scheme (see (\ref{eq:select})),
	there is $z_\ell\in Q_\ell$ with $\varphi(z_\ell) \in B_\ell(y)$. Choosing a
	convergent subsequence of $(z_\ell)$, if necessary, we may assume that
	$z = \lim_{\ell\to\infty}z_\ell$. By construction, $z\in Q_\infty$, and since
	$\lim_{\ell\to\infty}\diam(B_\ell(y)) = 0$ we conclude that $\lim_{\ell\to\infty} \varphi(z_\ell) = y$.
	Finally $\varphi$ is continuous, and therefore $y = \varphi(z)$. Hence $y\in
	\varphi(Q_\infty)$.
\end{proof}

For the inverse inclusion, we need to introduce an additional assumption, namely that
$\varphi^{-1}(A_Q) \subset A_Q$. This is automatically satisfied in the case where
$\varphi$ is a homeomorphism. Moreover if
$A_k$ is attracting and $A_k=A_Q$ then $A_Q$ is backward invariant.
These observations justify this assumption.

\begin{lemma}
	\label{lem:A_QinQ_inf}
	Suppose that $\varphi^{-1}(A_Q) \subset A_Q$, then
	\[
	A_Q \subset Q_\infty.
	\] 
\end{lemma}

This proof is identical to the proof of Lemma 3.2 in \cite{DH97}. Thus, we will not
restate it here.

We summarize the convergence result in the following
\begin{proposition}
Suppose that the relative global attractor $A_Q$ satisfies
$\varphi^{-1}(A_Q) \subset A_Q$. Then
\[
A_Q = Q_\infty.
\]
\end{proposition}

\subsection{Approximation of Attracting Sets}
\label{ssec:approx_att}
Observe that by construction of $\varphi$, the set $A_Q$ defined in
Section \ref{ssec:subdiv} \emph{contains} the one-to-one image $A_k$ of
the invariant set $A$ of $\Phi$.  We now show that by using sufficiently
high powers of $\Phi$ we can actually \emph{approximate} a one-to-one
image of $A$ if $A$ is attracting. 

Therefore,
we now assume that $A$ is an attracting set, that is, $A$ attracts
all bounded sets within a neighborhood $U$ of $A$.
Moreover we assume that the set $Q$ is chosen in such a way that
\begin{equation}\label{eq:EQU}
A_k \subset Q \quad \mbox{and} \quad E(Q)\subset U.
\end{equation}
Hence, for every $x\in Q$, $\Phi^j(E(x))$ will eventually approach the
attracting set $A$ for $j\to \infty$. However, this alone does not guarantee
that $A_k$ is also an attracting set for the dynamical system $\varphi$.
For instance, it may be the case that for a certain $\bar x \in Q$ one has a
``spurious fixed point'' in the sense that
\[
\bar x = \varphi (\bar x)
\]
although $\Phi(E(\bar x))\not= E(\bar x)$ may be closer to $A$ than $E(\bar x)$.

In order to overcome this problem we now define for $m\ge 1$ the continuous maps
\begin{equation}\label{eq:phij}
\varphi_m = R\circ \Phi^m \circ E
\end{equation}
and denote the corresponding relative global attractors  by
$A^m_Q$.
\begin{remark}
\label{rmk:Phim}
Observe that $A$ is an invariant set for $\Phi^m$ for every $m$
and therefore we can still use $R$ as the restriction in our construction of the
dynamical system $\varphi^m$.
\end{remark}

\begin{lemma}
$A_k \subset A^m_Q$ for all $m\ge 1$.
\end{lemma}

\begin{proof}
Since $\Phi(A) = A$ we have $\varphi_m(A_k) = A_k$ for $m\ge 1$.
Moreover $A_k \subset Q$ (see (\ref{eq:EQU})), and Lemma~\ref{lem:subset} implies that
$A_k \subset A^m_Q$.
\end{proof}

Define
\[
A^\infty_Q = \bigcap_{m\ge 1} A^m_Q.
\]

Obviously $A_k \subset A^\infty_Q$. Moreover we have

\begin{proposition}\label{prop:attracting}
$A_k = A^\infty_Q$.
\end{proposition}

\begin{proof}
Suppose that $x\in A^\infty_Q \setminus A_k$. As $A_k$ is compact, this
implies $\mbox{dist}(x,A_k) = \epsilon > 0$.
As $A$ is compact, $R$ is continuous and $A_k = R(A)$, there is $\delta > 0$  such that 
\[
\mbox{dist}(u,A) < \delta \Rightarrow \mbox{dist}(R(u),A_k)) < \frac {\epsilon}{2}.
\]
Set $V = E(Q) \subset U$ by assumption (see (\ref{eq:EQU})).
Since $A$ is attracting and $V$ is a compact set within $U$
we can find an $m\ge 1$ such that
\[
h(\Phi^m(V),A) < \delta,
\]
where $h$ is the Hausdorff distance. By our choice of $\delta$ it follows that
\[
h(R(\Phi^m(V)),A_k) = 
h(\varphi_m(Q),A_k) < \frac {\epsilon}{2}.
\]
Thus,
\[
x\not\in \varphi_m(Q) \Rightarrow x\not\in A^m_Q \Rightarrow x\not\in A^\infty_Q
\]
yielding a contradiction.
\end{proof}

\begin{remark}
Roughly speaking Proposition \ref{prop:attracting} states that it will be
possible to approximate an attracting set for $\Phi$ if we perform the computations with
appropriately high iterates of $\Phi$.
\end{remark}

\section{Numerical Realization for Delay Differential Equations}
\label{sec:num_real_DDE}
As one important setting where finite dimensional dynamical phenomena occur in
infinite dimensional Banach spaces,
we consider delay differential equations with constant time delay $\tau > 0$.
More precisely we consider equations of the form 
 \begin{equation}
	\dot y(t) = g(y(t),y(t-\tau)),
	\label{eq:DDE}
\end{equation}
where $y(t)\in\R^n$ and $g:\R^n \times \R^n \rightarrow \R^n$ is a smooth map. 
Following \cite{HL93}, we denote by $\cC = C([-\tau, 0], \R^n)$ the (infinite dimensional)
state space of the dynamical system  (\ref{eq:DDE}). Equipped
with the maximum norm, $\cC$ is a Banach space.

Let $y_u(t)$ be the trajectory generated by (\ref{eq:DDE}) with the initial condition
$u \in \cC$. Then the flow $\Phi^s: \cC \to \cC$ of (\ref{eq:DDE}) is given by
\[
u \mapsto \Phi^s (u), \mbox{ where $\Phi^s(u)(t) = y_u( s-t)$ for $t\in [-\tau, 0]$.}
\] 
Next we fix $\omega > 0$ and consider the corresponding \emph{time-$\omega$-map} $\Phi^\omega: \cC\to \cC$ 
as our dynamical system. That is, we set
\begin{equation}
\label{eq:PhiY}
\Phi = \Phi^\omega \quad \mbox{and} \quad Y = \cC
\end{equation}
in our abstract dynamical system (\ref{eq:DS}).

In order to numerically realize the construction of the map $\varphi = R\circ
\Phi\circ E$ described in Section~\ref{sec:comp_emb_att}, we have to work on
three tasks: the implementation of $E$, of $R$, and of $\Phi^\omega$ 
respectively. For the latter we will rely on standard methods for forward
time integration of DDEs \cite{bellen2013numerical}. The map $R$ will be realized
on the basis of Theorem~\ref{thm:R05} and Remark~\ref{rmk:ext_R05} by an
appropriate choice of observables. For the numerical construction of the embedding $E$
we will employ a
bootstrapping method that re-uses results of previous computations. This way we will
in particular guarantee that the identities in \eqref{eq:condRE} are at least approximately
satisfied.

From now on we assume that upper
bounds for both the box counting dimension $d$ and the thickness exponent $\sigma$
are available.
This allows us to fix $k > 2(1+\sigma)d$ according to Theorem~\ref{thm:R05}.

\subsection{Numerical Realization of $R$}
\label{ssec:num_real_R}

For the definition of $R$ we have to specify the time span $\omega$ and
appropriate corresponding observables.
In the case of a scalar equation ($n=1$) we choose the observable $f$ to be
\[
f(u) =u(-\tau).
\]
Thus, in this case the restriction $R$ is simply given by
\[
R = D_k[f,\Phi] (u) = (u(-\tau),\Phi(u(-\tau)),\ldots,\Phi^{k-1}(u(-\tau)))^T.
\]
The time span $\omega$ (see (\ref{eq:PhiY})) is defined to be a natural fraction of $\tau$, that is
\begin{equation}
\label{eq:TK}
\omega = \frac{\tau}{K} \mbox{ for } K\in\N.
\end{equation}

\begin{remark}
\label{rmk:K}
\begin{itemize}
\item[(a)] Observe that a natural choice for $K$ in (\ref{eq:TK}) would be $K = k-1$
for $k>1$.
That is, for each evaluation of $R$  the observable would be applied
to a function $u:[-\tau,0] \rightarrow \R$ at $k$ equally distributed time steps within the interval $[-\tau,0]$.
\item[(b)] As described in Section~\ref{ssec:approx_att} (see Remark~\ref{rmk:Phim})
we will frequently replace
$\Phi$ by $\Phi^m$ ($m > 1$) in order to speed up the convergence
towards the invariant sets $A$ resp.\ $A_k$. For an illustration of this procedure
see Figure~\ref{fig:RKm}.
\end{itemize}
\end{remark}

\begin{figure}[htb]
    \centering
        \includegraphics[width=0.84\textwidth]{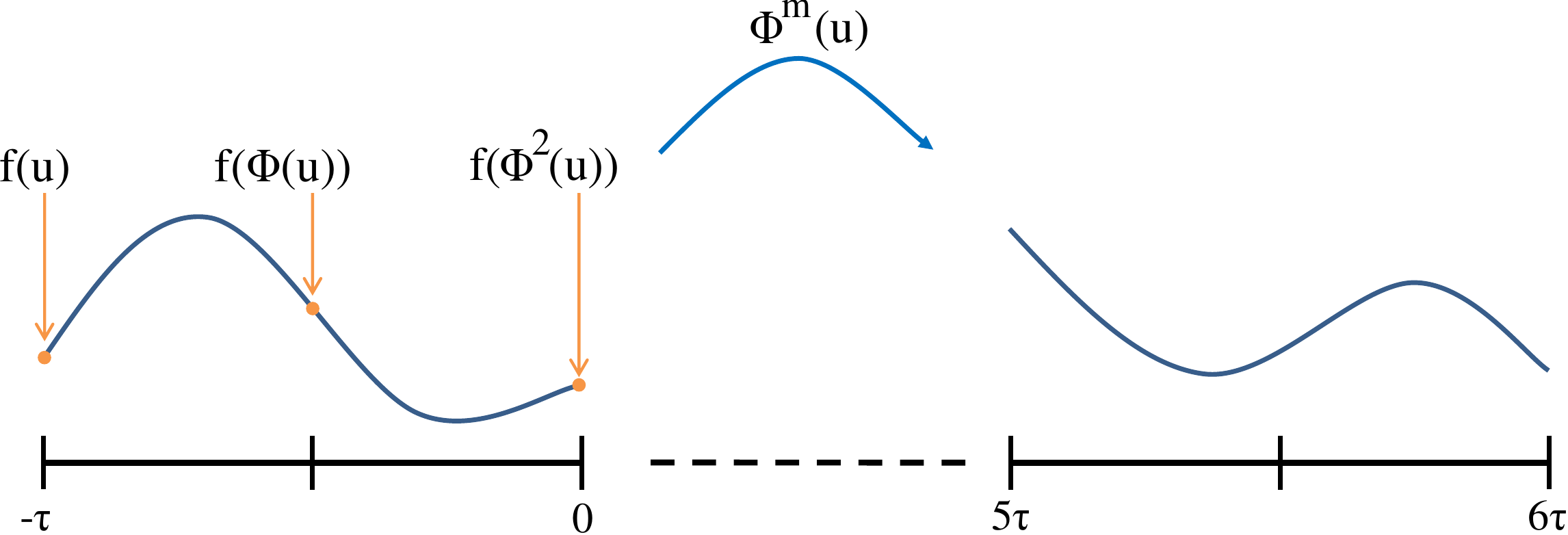} 
    \caption{Numerical realization of the restriction $R$ for $n=1$, $K=2$ and $m=6K$.}
    \label{fig:RKm}
\end{figure}

For the numerical analysis of systems of delay differential equations ($n>1$) we make use of Remark~\ref{rmk:ext_R05}
as follows: For each component $u_j$ of $u$ we define a separate observable $f_j$ $(j=1,\ldots,n)$ by
\begin{equation}
\label{eq:fjdef}
f_j(u) = u_j(\nu_j)\quad \mbox{for a $\nu_j \in [-\tau ,0]$,}
\end{equation}
and choose different time spans
\begin{equation}
\label{eq:defomega}
\omega_j = \frac{\tau}{K_j} \mbox{ for } K_j \in \N
\end{equation}
accordingly.

Finally, we note that also more general constructions for the restriction
$R:\cC\to \R^k$ can be employed. In fact, by virtue of Theorem
\ref{thm:HK99}, for any $k$ that is sufficiently large for the delay
embedding construction, an arbitrary linear map $\cC\to \R^k$ will
generically be one-to-one on $A$. Therefore almost every linear
combination of trajectory points computed during forward integration can
be used for the construction of the map $R$.

\subsection{Numerical Realization of $E$}
\label{ssec:num_real_E}
In the application of the subdivision scheme for the computation of
the relative global attractor $A_k$
described in Section \ref{ssec:subdiv} one has to perform
the selection step 
\[
    \cB_\ell=\left\{B\in\hat\cB_\ell : \exists \hat B\in\hat\cB_\ell
      ~\mbox{such that}~\varphi^{-1}(B)\cap\hat B\ne\emptyset\right\}
\]
(see \eqref{eq:select}). Numerically this is realized as follows: At first
$\varphi$ is evaluated for a large number of
test points $z_j^k\in B_k$ for each box $B_k \in \hat\cB_\ell$. Then
a box $B_m$ is kept in the collection
$\cB_\ell$ if it is hit by (at least) one of the images $\varphi(z_{\tilde j}^{\tilde k})$.

\begin{remark}
In practice the test points $z_j^k\in B_k$ can be chosen according to several different
strategies: In low dimensional problems one can choose them
from a regular grid within each box $B_k$. Alternatively one can
select the test points from the boundaries of the boxes. In our
computations we have sampled a fixed number of test points
from each box at random with respect to a uniform distribution.
\end{remark}

For the evaluation of $\varphi = R\circ \Phi \circ E$ at a test point $z$ we need to 
define the image $E(z)$, that
is, we need to generate adequate initial conditions for the
forward integration of the DDE (\ref{eq:DDE}). 
In the first step of the subdivision procedure, when no information
on $A$
is available, we proceed as follows.
In the case of a scalar delay equation, that is $n=1$, 
we construct a piecewise linear function $u = E(z)$, where
\begin{equation}
 u(t_i) = z_i,
\end{equation}
for $t_i = -\tau + i\cdot \omega,\ i=0,\dots,k-1$.
Observe that by this choice of $E$ and $R$ the condition $R\circ E(z) = z$ is satisfied for
each test point $z$ (see \eqref{eq:condRE} and Remark~\ref{rmk:K} (a)).

For $n>1$ we proceed analogously and distribute the components of $z\in \R^k$
to the components $u_j$ of $u = E(z)\in \R^n$ according to \eqref{eq:fjdef}
and \eqref{eq:defomega}.
Also in this case the condition $R\circ E(z) = z$ still holds.

In the following steps of the subdivision procedure we proceed as follows:
Note that if $B\in \cB_\ell$, then, by the selection step, there must have been a $\hat B\in\cB_{\ell -1}$ 
such that $R(\Phi^\omega(E(\hat z)))\in B$ for at least one test point $\hat z\in \hat B$.
Therefore, we can use the information from the computation of $\Phi^\omega(E(\hat z))$
to construct an appropriate $E(z)$ for each test point $z\in B$.

More concretely, in every step of the subdivision procedure, for every set
$B \in \cB_\ell$ we keep additional information about the points
$\Phi^\omega(E(\hat z))$ that
were mapped into $B$ by $R$ in the previous step. In the simplest case, we store $k_i\geq 1$
additional equally distributed function values for each interval $(-\tau+(i-1) \omega,-\tau+i\omega)$ for $i=1,\dots,k-1$.  
When $\varphi(B)$ is to be computed using test points from $B$, we first use the points in $B$
for which additional information is available and generate the corresponding initial value functions 
via spline interpolation. Note that the more information we store, the smaller the error 
$\|\Phi^\omega(E(\hat z))-E(z)\|$ becomes for $z = R(\Phi^\omega(E(\hat z)))$. That is, we
enforce an approximation of the identity $E\circ R(u) = u$ for all $u\in A$ (see \eqref{eq:condRE}).

If the additional information is available only for a 
few points in $B$, we generate new test points in $B$ at random and construct corresponding trajectories
by piecewise linear interpolation.

\section{Numerical Results}
\label{sec:numex}
In this section we present results of computations carried out for three
different delay differential equations. In each case, $u(t)$ is scalar,
although for the DDE considered in Section \ref{ssec:Arneodo} the
problem is recast into a three-dimensional form in order to obtain a
first-order equation.

\subsection{The Modified Wright Equation}
\label{ssec:Wright}
As the first example, we consider a modification of the Wright equation,
\begin{equation}\label{eq:wright}
\dot u(t) = -\alpha\cdot u(t-1)\cdot [1-u^2(t)].
\end{equation}
In \cite{HL93} it has been shown that the stationary solution $u_0(t) \equiv 0$
undergoes a supercritical Hopf bifurcation at
$\alpha = \pi/2$.  Thus, \eqref{eq:wright} possesses a stable
periodic solution for $\alpha > \pi/2$ -- at least locally.
In our computations we set $\alpha = 2$, choose the embedding
dimension $k=5$, and approximate the relative global attractor $A_Q\subset \R^5$
for $Q = [-2,2]^5$, see Figure~\ref{fig:wright_unstable_manifold}. Here the set $A_Q$
consists of a reconstruction of the two-dimensional unstable manifold of $u_0 \equiv 0$
which accumulates on a stable periodic orbit at its boundary.
\begin{figure}[ht]
\begin{center}
\includegraphics[width = .9\textwidth]{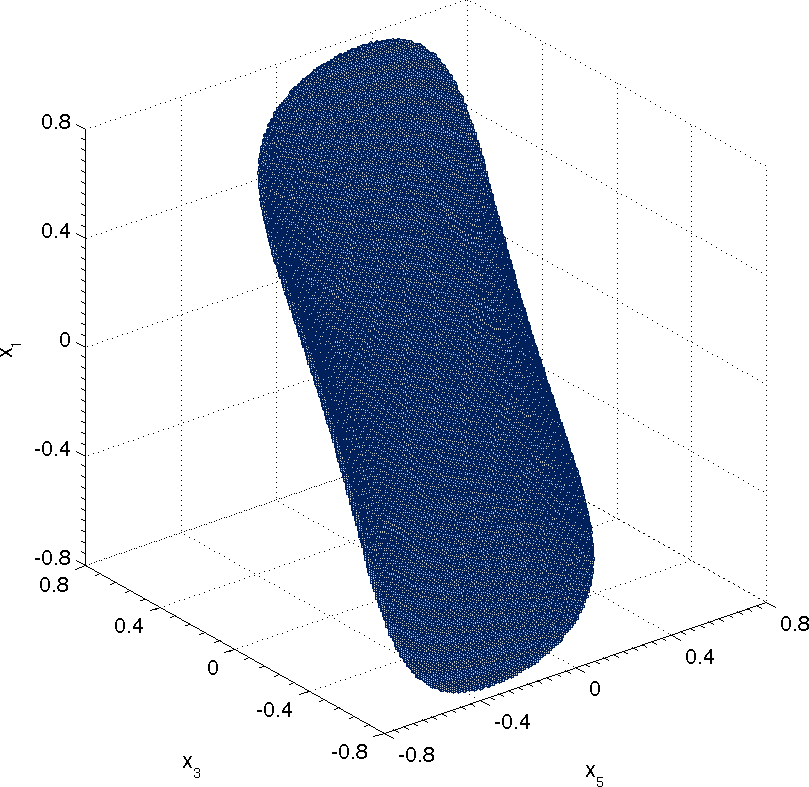}
\end{center}
\caption{Three-dimensional projection of an approximation of the relative global attractor $A_Q$
within $Q = [-2,2]^5$ for equation \eqref{eq:wright} after $\ell = 45$ subdivision steps 
(embedding dimension $k=5$ and iteration exponent
$m=16$, see Section~\ref{ssec:num_real_R})}
\label{fig:wright_unstable_manifold}
\end{figure}

In Figure \ref{fig:wright_subdivision} we show box coverings of the 
reconstructed periodic solution itself. These have been obtained by removing
a small open neighborhood $U$ of the origin from $Q = [-2,2]^5$ and
computing $A_{\widetilde Q}$ for $\widetilde Q = Q \setminus U$.

\begin{figure}[ht]
\begin{minipage}{0.49\textwidth}
 \includegraphics[width = \textwidth]{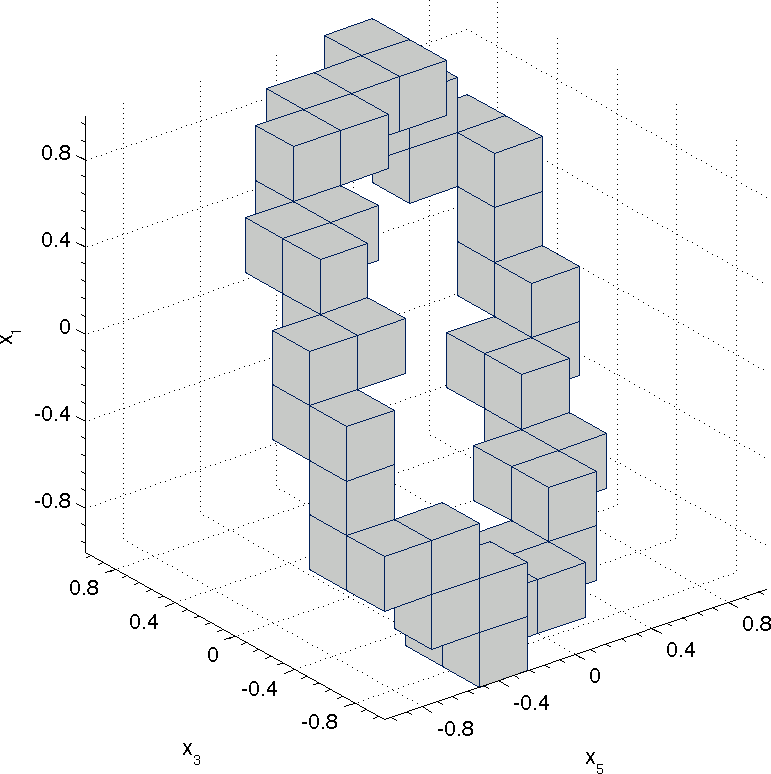}\\
 \centering \scriptsize{(a) $\ell = 20$}
\end{minipage}
\hfill
\begin{minipage}{0.49\textwidth}
\includegraphics[width = \textwidth]{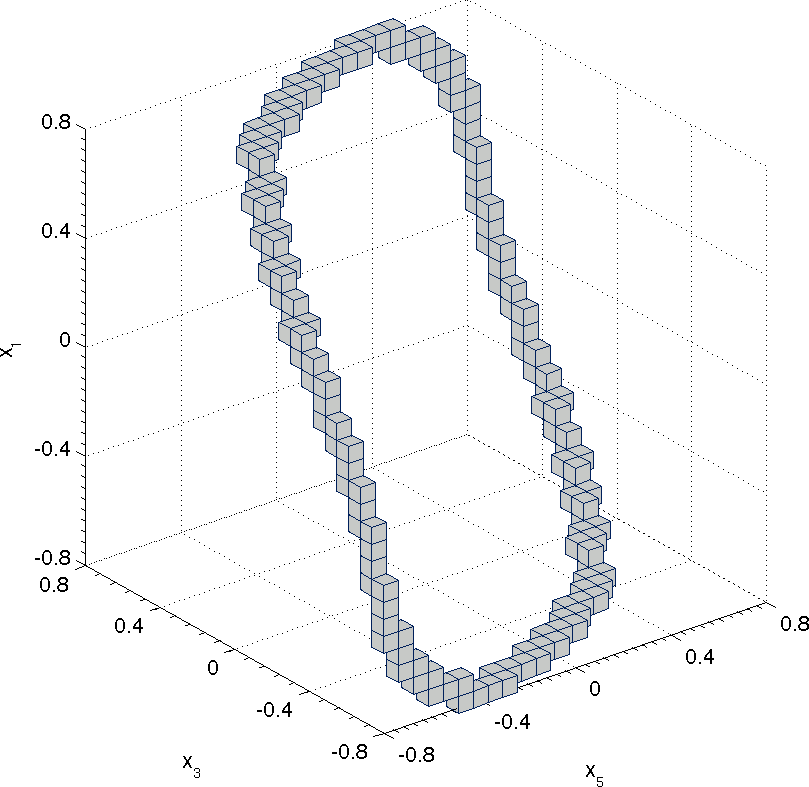}\\
\centering \scriptsize{(b) $\ell = 30$}
\end{minipage}\\[1em]
\begin{minipage}{0.49\textwidth}
 \includegraphics[width = \textwidth]{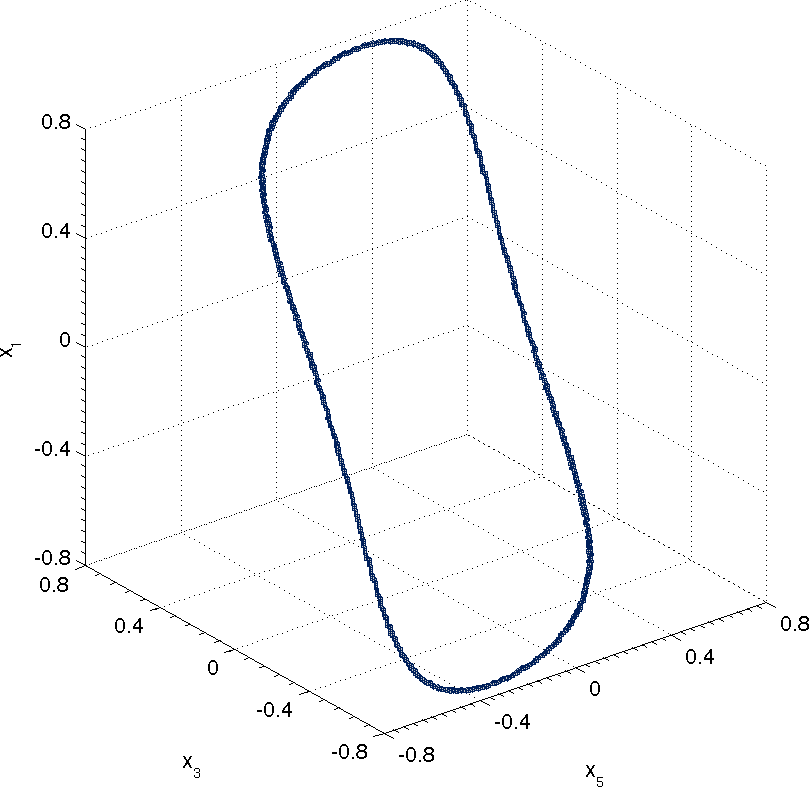}\\
 \centering \scriptsize{(c) $\ell = 45$}
\end{minipage}
\hfill
\begin{minipage}{0.49\textwidth}
\includegraphics[width = \textwidth]{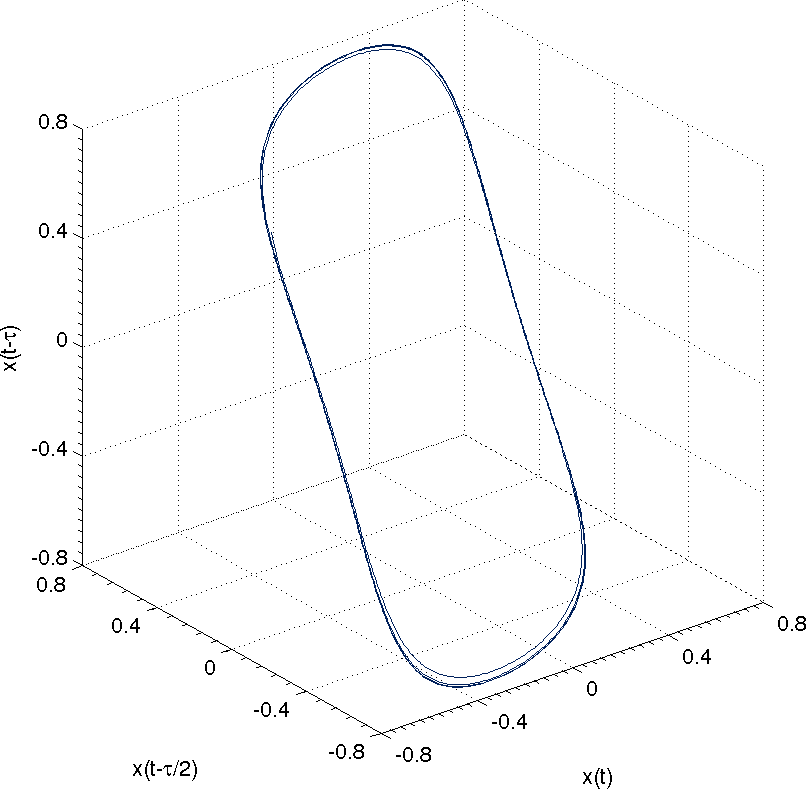}\\
\centering \scriptsize{(d) Simulation}
\end{minipage}
\caption{(a)-(c) Three-dimensional projections of successively finer coverings of the relative global
	attractor $A_{\widetilde Q}$ which corresponds to a reconstruction of a periodic
	orbit of \eqref{eq:wright}. (d)~Periodic orbit computed by direct simulation.}
\label{fig:wright_subdivision}
\end{figure}

\subsection{The Arneodo System with Delay}
\label{ssec:Arneodo}
The second example is a modification of the Arneodo system \cite{Arneodo1982}
where a delay is introduced in the first order derivative of $u$,
\begin{equation*}
 \frac{d^3u}{dt^3}(t) + \frac{d^2u}{dt^2}(t) + 2\frac{du}{dt}(t-\tau) - \alpha u(t) +u^2(t) = 0.
\end{equation*}
This equation has been introduced and analyzed in \cite{SahaiV09}.
In our computations we use the equivalent reformulation as a first-order system
\begin{align}
\nonumber \dot u_1 &= u_2,\\
\dot u_2 &= u_3, \label{eq:DDE_Arneo3D} \\
\nonumber \dot u_3 &= -u_3 - 2u_2(t-\tau)+\alpha u_1 - u_1^2.
\end{align}
The undelayed system (i.\ e.\ (\ref{eq:DDE_Arneo3D}) with $\tau = 0$) has
been studied extensively.
It possesses the equilibria $O_1 =
(0,0,0)$ and $O_2 = (\alpha,0,0)$, the latter is asymptotically stable for $\alpha
< 2$. At $\alpha = 2$ the equilibrium $O_2$ undergoes a supercritical
Hopf bifurcation (cf. \cite{KrOs1999}).
For values of $\alpha$ which are slightly larger than two, points on the two-dimensional
unstable manifold of $O_2$ converge to the corresponding limit cycle on the
branch of periodic solutions. That is, topologically speaking, we have the same situation as
in Figure~\ref{fig:wright_unstable_manifold}.

For the delayed (i.\ e. $\tau>0$) equation, the Hopf bifurcation occurs at
decreasing values of $\alpha$ for increasing values of $\tau$. For fixed
$\alpha = 2.5$, the amplitude of the limit cycle grows with increasing values of
$\tau$ and loses its stability in a period-doubling bifurcation at
$\tau\approx 0.11$ \cite{SahaiV09}. Our purpose is to investigate the
structure of the relative global attractor right after the occurrence of the
period-doubling bifurcation.
Concretely we set $\alpha = 2.5$, $\tau = 0.13$, choose the embedding dimension $k=5$,
and approximate the relative global attractor $A_Q\subset \R^5$
for $Q = [-1, 5] \times [-4, 2] \times [-4, 4] \times [-4, 4] \times [-4, 4]$. This way we compute a reconstruction of the two-dimensional
unstable manifold of the origin which accumulates on a period-doubled limit cycle. 

In this example we have made use of Remark~\ref{rmk:ext_R05} in our
numerical realization. Concretely
we have chosen $\omega = \tau / 2$ and
the following three observables (see \eqref{eq:fjdef} and \eqref{eq:defomega})
\begin{eqnarray*}
f_1(u) & = & u_2(-\tau), \quad k_1 = 3, \quad K_1 = 2,\\
f_2 (u) & = & u_1(0), \quad k_2 = 1,\\
f_3(u) & = & u_3(0), \quad k_3 = 1.
\end{eqnarray*}

Thus, the restriction $R$ can be written as
\[
	R(u) = 
		(u_2(-\tau),u_2(-\tau / 2), u_2(0), u_1(0), u_3(0))^T
\]
Observe that $R:\cC \rightarrow \R^5$ is linear and therefore also Theorem~\ref{thm:HK99}
could be used in order to justify this construction.
The corresponding reconstructions of the relative global attractor are shown in 
Figure \ref{fig:res_arneodo1}.

\begin{figure}[ht]
\begin{minipage}{0.49\textwidth}
 \includegraphics[width = \textwidth]{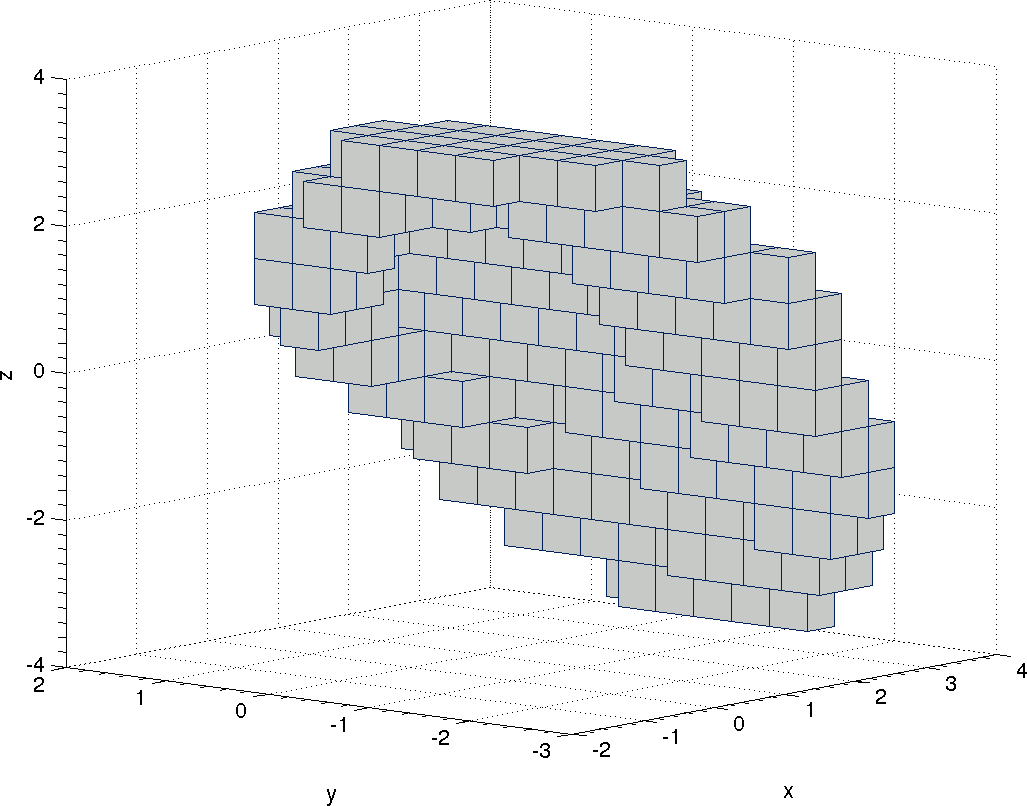}\\
 \centering \scriptsize{(a) $\ell = 20$}
\end{minipage}
\hfill
\begin{minipage}{0.49\textwidth}
\includegraphics[width = \textwidth]{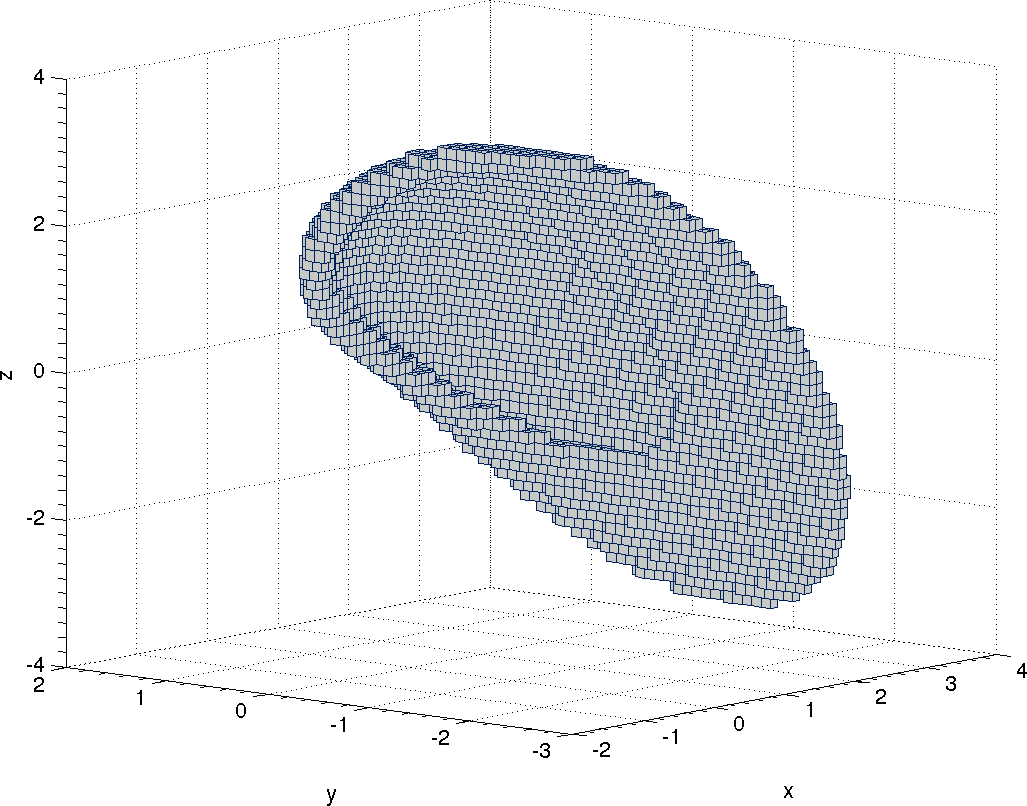}\\
\centering \scriptsize{(b) $\ell = 30$}
\end{minipage}\\[1em]
\begin{minipage}{0.49\textwidth}
 \includegraphics[width = \textwidth]{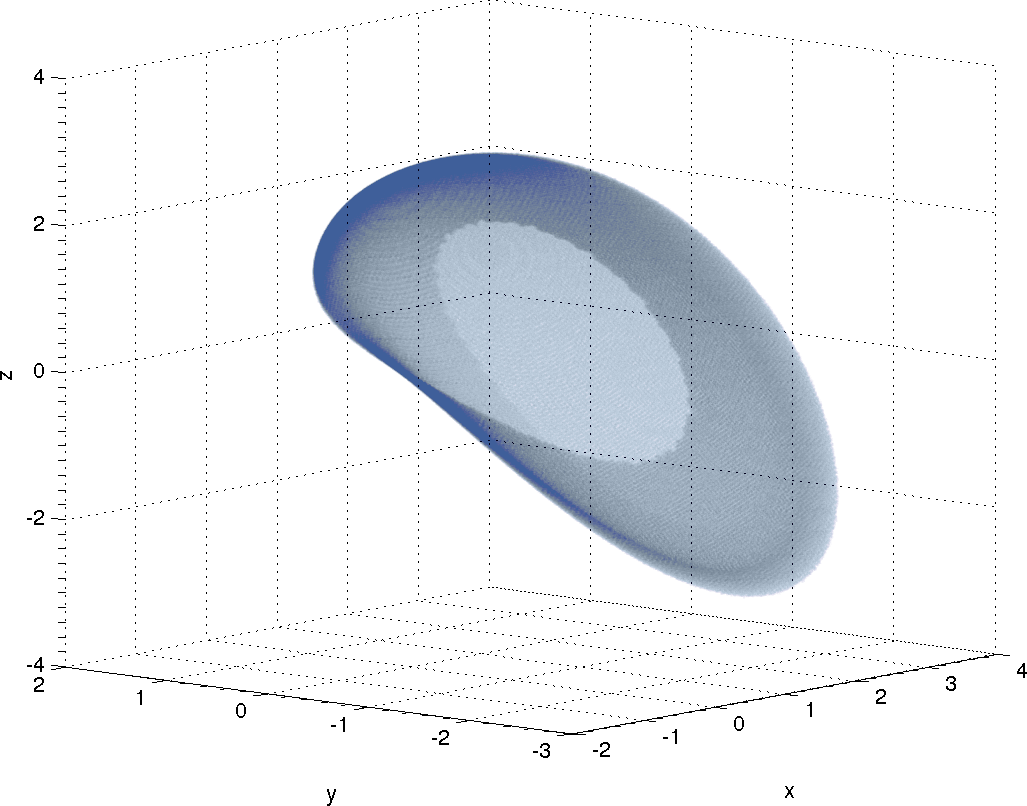}\\
 \centering \scriptsize{(c) $\ell = 45$}
\end{minipage}
\hfill
\begin{minipage}{0.49\textwidth}
\includegraphics[width = \textwidth]{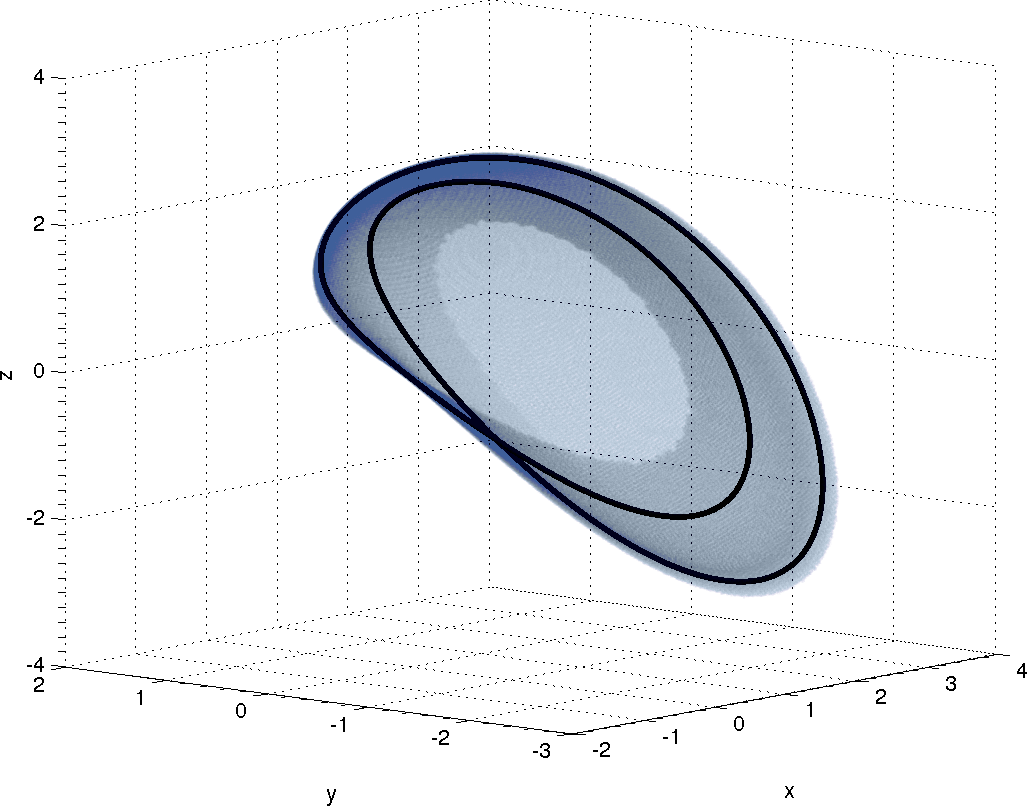}\\
\centering \scriptsize{(d) $\ell = 45$ with the period-doubled orbit computed by direct simulation}
\end{minipage}
\caption{(a)-(d) Successively finer coverings of the relative global attractor for the
        Arneodo DDE (\ref{eq:DDE_Arneo3D}) after $\ell$ subdivision steps
        ($\alpha = 2.5$, $\tau = 0.13$,
        embedding dimension $k=5$ and iteration exponent $m = 15$; see Section~\ref{ssec:num_real_R}).}
\label{fig:res_arneodo1}
\end{figure}

Observe that after the period doubling bifurcation the relative global
attractor contains a Moebius strip with the period-doubled periodic solution at its boundary.
Thus, in the course of the period doubling
bifurcation there has to occur a significant change of the geometry of the
unstable manifold at its boundary so that it can accommodate the
period-doubled solution. In fact, the corresponding mechanism has
been analyzed analytically already in 1984 by Crawford and
Omohundro \cite{Crawford84}. It turns out that at the period doubling 
the unstable manifold starts to wrap itself ``infinitesimally'' around the
unstable periodic solution. In a corresponding Poincar\'e section 
this becomes a spiraling behavior with very sharp curvature, and
we analyze this behavior at the reconstruction in Figure~\ref{fig:res_arneodo2}
(see also Figure~16 in \cite{Crawford84}).
However, we expect that one would have to choose a much higher
resolution (i.\ e.\ higher number of subdivisions) in order to reveal this
dynamical behavior more clearly.

\begin{figure}[ht]
\begin{minipage}{0.49\textwidth}
 \includegraphics[width = \textwidth]{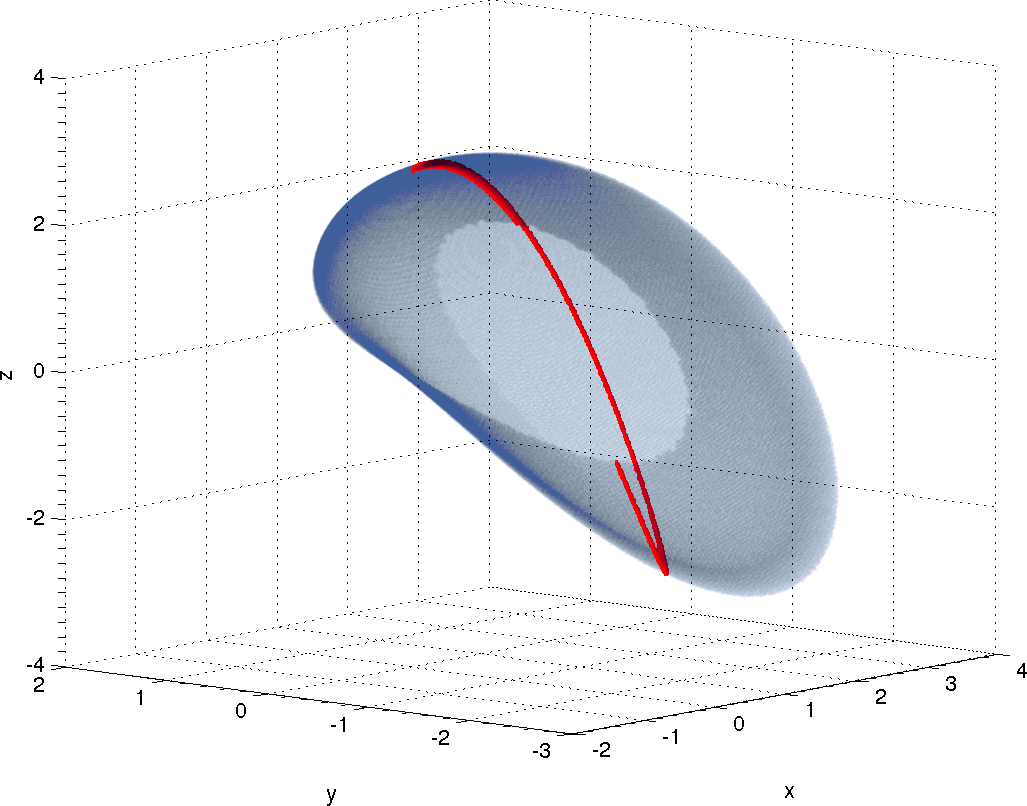}\\
 \centering \scriptsize{(a) Reconstructed unstable manifold.
 Red boxes (at $y=0$) illustrate the location of the Poincar\'e section.}
\end{minipage}
\hfill
\begin{minipage}{0.49\textwidth}
\includegraphics[width = \textwidth]{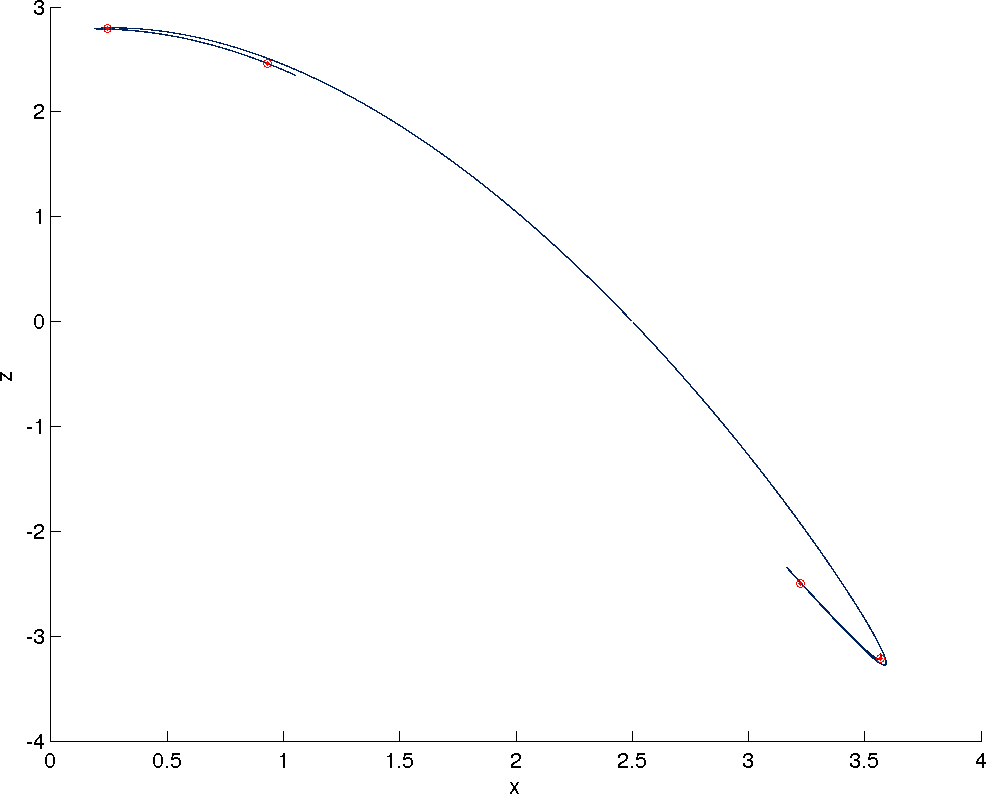}\\
\centering \scriptsize{(b) Poincar\'e section. Red circles mark the intersection
with the period doubled periodic orbit.}
\end{minipage}\\[1em]
\begin{minipage}{0.49\textwidth}
 \includegraphics[width = \textwidth]{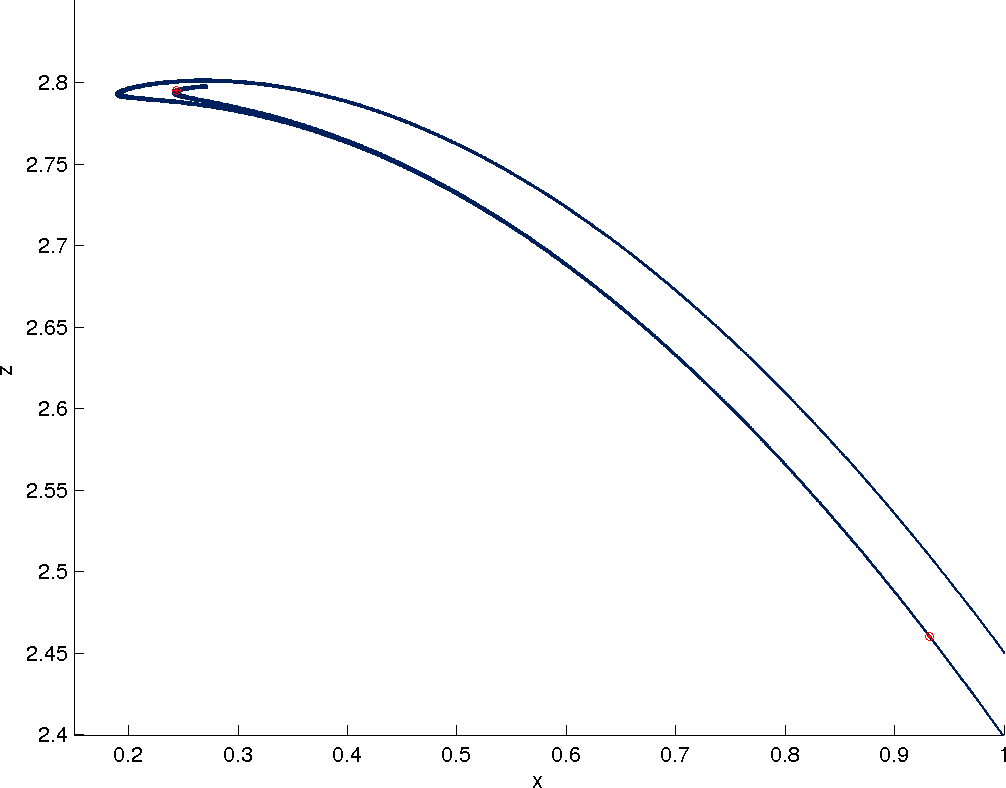}\\
 \centering \scriptsize{(c) Zoom into the upper left corner of (b)}
\end{minipage}
\hfill
\begin{minipage}{0.49\textwidth}
\includegraphics[width = \textwidth]{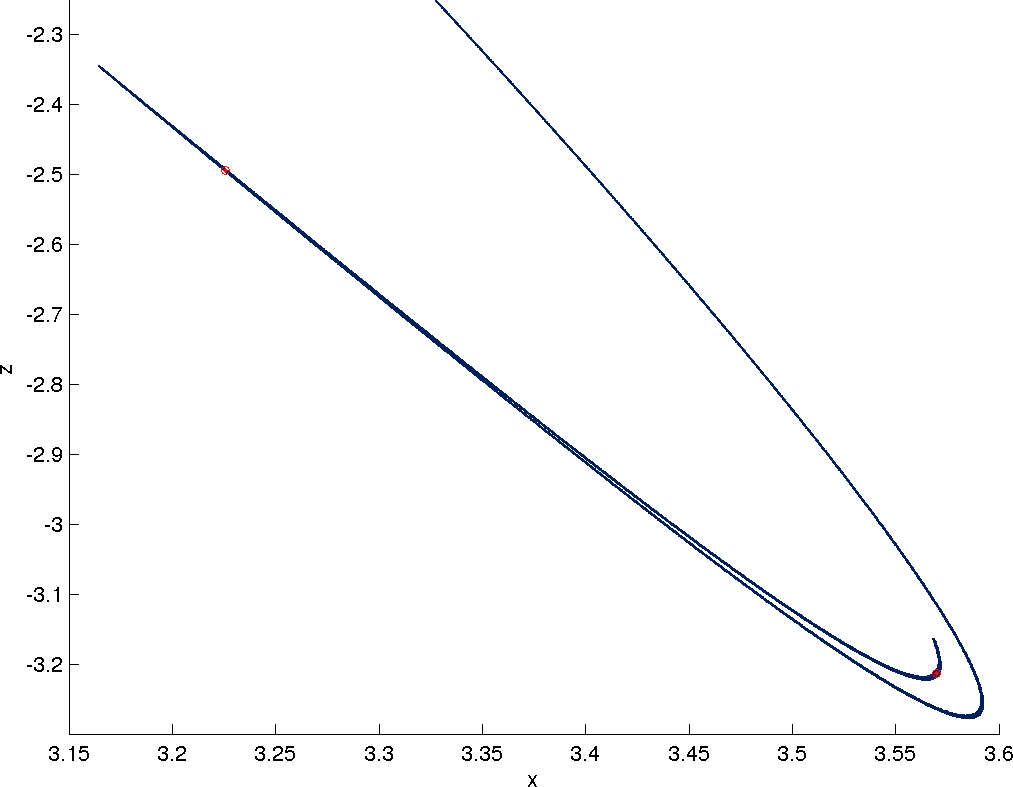}\\
\centering \scriptsize{(d) Zoom into the lower right corner of (b)}
\end{minipage}
\caption{Detailed illustration of a Poincar\'e section through the relative
global attractor for \eqref{eq:DDE_Arneo3D}.}
\label{fig:res_arneodo2}
\end{figure}

\subsection{The Mackey-Glass Equation}
\label{ssec:MG}
Our final example is the well-known delay differential 
equation introduced by Mackey and Glass in 1977 \cite{mackey1977}, namely
\begin{equation}\label{eq:mg}
  \dot u(t) = \beta \frac{u(t-\tau)}{1+u(t-\tau)^\eta}-\gamma u(t),
\end{equation}
where we choose $\beta = 2, \gamma = 1$, $\eta = 9.65$, and $\tau = 2$. This equation is a
model of blood production, where $u(t)$ represents the concentration of blood at
time $t$, $\dot u(t)$ represents production at time $t$ 
and $u(t-\tau)$ is the concentration at an earlier time, when the request for more blood is made.
Direct numerical simulations indicate that the dimension of the corresponding
attracting set is approximately $d=2$. Thus, we choose the embedding dimension
$k=7$, and approximate the relative global attractor $A_Q$ for $Q = [0,1.5]^7 \subset \R^7$. 
In Figure \ref{fig:mg_subdivision} we show projections of the coverings
obtained after $\ell=28,42$ and $63$ subdivision steps as well as a direct simulation.

\begin{figure}[ht]
\begin{minipage}{0.49\textwidth}
 \includegraphics[width = \textwidth]{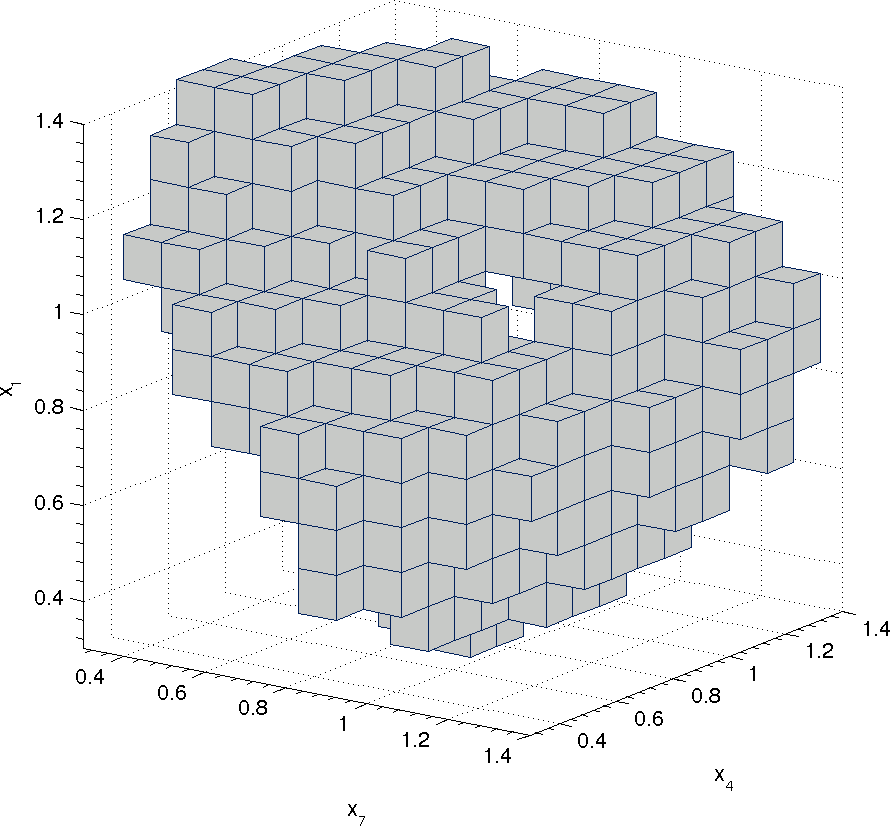}\\
 \centering \scriptsize{(a) $\ell = 28$}
\end{minipage}
\hfill
\begin{minipage}{0.49\textwidth}
\includegraphics[width = \textwidth]{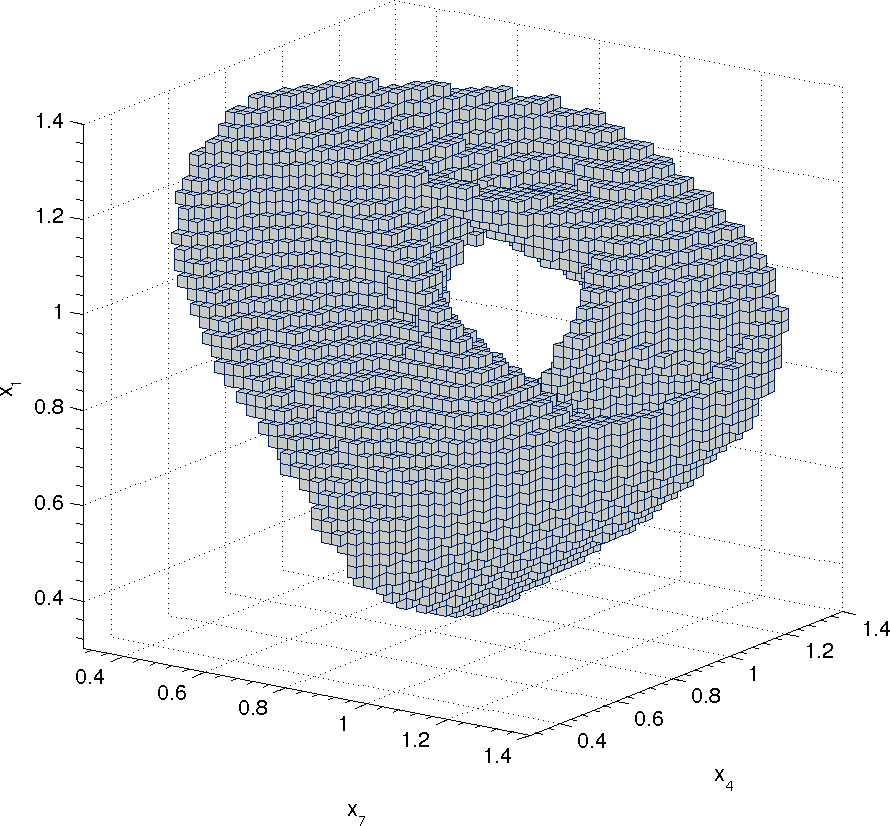}\\
\centering \scriptsize{(b) $\ell = 42$}
\end{minipage}\\[1em]
\begin{minipage}{0.49\textwidth}
 \includegraphics[width = \textwidth]{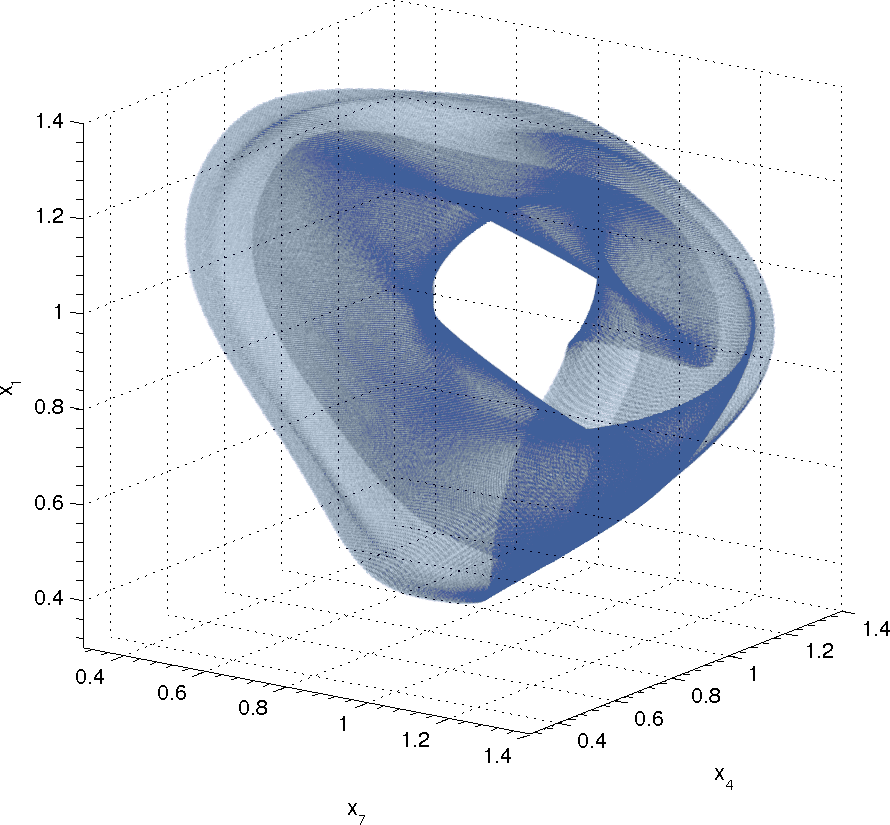}\\
 \centering \scriptsize{(c) $\ell = 63$}
\end{minipage}
\hfill
\begin{minipage}{0.49\textwidth}
\includegraphics[width = \textwidth]{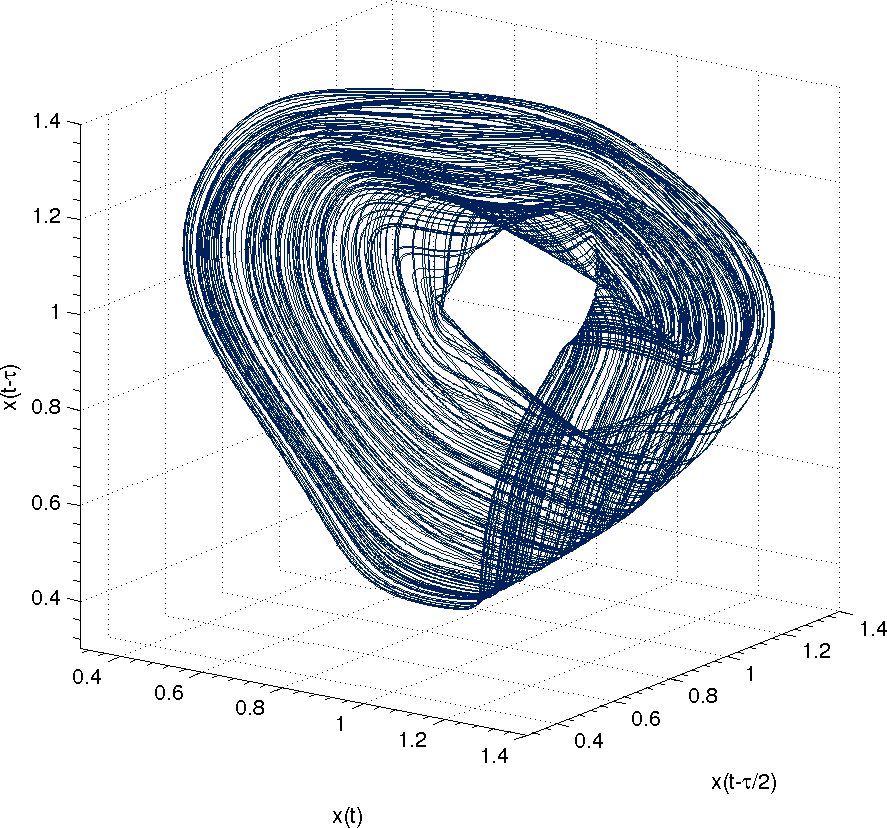}\\
\centering \scriptsize{(d) Simulation}
\end{minipage}
\caption{(a)-(c) Successively finer coverings of a relative global
	attractor after $\ell$ subdivision steps for the Mackey-Glass
	equation \eqref{eq:mg}; (d) direct simulation.}
\label{fig:mg_subdivision}
\end{figure}

Finally we conclude this section with an outlook and show a corresponding
invariant measure for the reconstructed Mackey-Glass attractor in
Figure \ref{fig:mg_comparison_measure}. This measure has been computed
with the software package GAIO \cite{DFJ-GAIO} which is based on the techniques
developed in \cite{DJ99}. However, a detailed investigation on how
to approximate invariant measures in infinite dimensional problems
efficiently using embedding theory will be done in future work.

\begin{figure}[ht]
\begin{center}
\includegraphics[width = 0.8\textwidth]{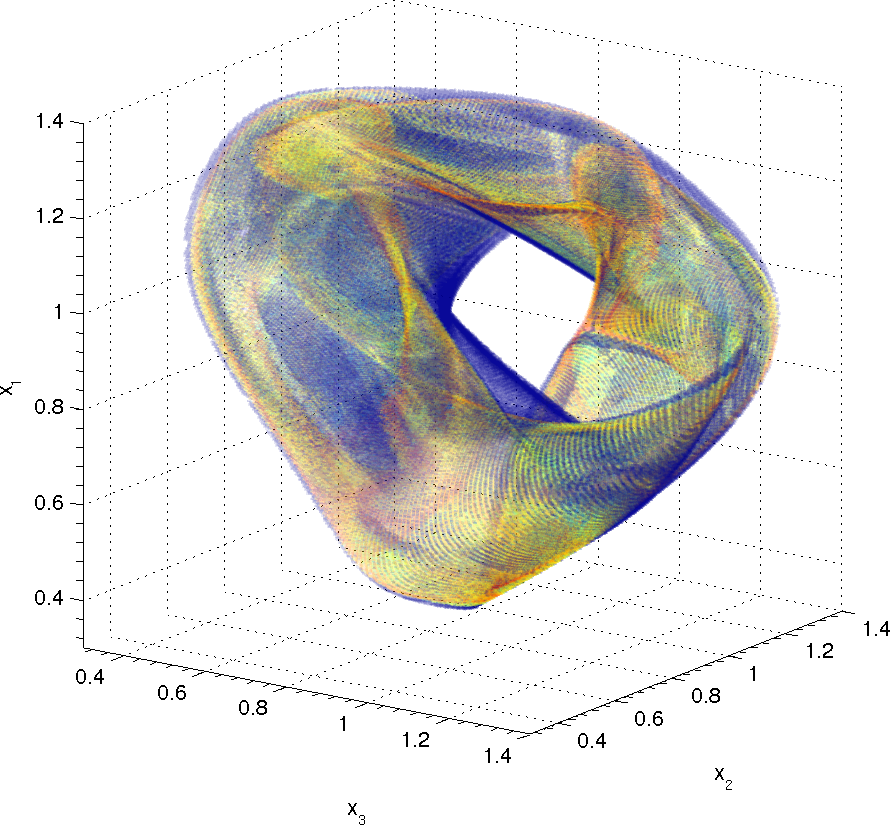}\\
\end{center}
\caption{Invariant measure for the reconstructed Mackey-Glass attractor. The density ranges from blue (low density) $\rightarrow$ green $\rightarrow$ yellow $\rightarrow$ red (high density). }
\label{fig:mg_comparison_measure}
\end{figure}

\section{Conclusion}
In contrast to the situation for finite dimensional dynamical systems, for
which there exists a wide range of advanced numerical tools, there are
currently only few options besides direct simulation for the numerical
computation of attractors of infinite dimensional dynamical systems generated e.~g.\ by
DDEs. In this paper we develop a general methodology for the computation
of finite dimensional compact attractors of infinite dimensional dynamical
systems, and illustrate its application to several DDEs.

Combining the delay embedding technique with a set-oriented method for the
computation of attractors, we obtain a flexible method for the 
analysis of infinite dimensional dynamical systems.  
More concretely, we use standard techniques for short-time simulation of
the system in question to approximate the infinite dimensional dynamics,
and employ the delay embedding technique on simulation results in order to
obtain a representation of the dynamics by a continuous map on a 
moderately sized space. This map, in turn, can be analyzed using
the subdivision scheme in order to compute a covering of an attractor. We
show that in this way one obtains sets that are in one-to-one
correspondence with the infinite dimensional attractor. 

The method proposed shares vital characteristics with its counterparts for
finite dimensional systems. Most importantly, the numerical effort
essentially depends on the dimension of the object to be computed, and
\emph{not} on the dimension of ambient space used for computations, that
is in the case of this paper, the dimension of the space used for the
delay embedding. In three examples, we have illustrated the suitability for
the computation of two-dimensional attractors. (In the case of the
Mackey-Glass equation, the attractor probably even has a box counting
dimension larger than $d=2.6$, see e.g. \cite{FAR82}.) Furthermore, due to
the set oriented nature of the underlying subdivision algorithm, the
method does not depend on any special geometric properties (besides
compactness) of the attractor.

\bibliographystyle{AIMS}
\bibliography{Takens-DDE}

\end{document}